\documentclass [12pt] {amsart}
\usepackage [latin1]{inputenc}
\usepackage[all]{xy}
\usepackage[english]{babel} 
\usepackage{amssymb, enumerate}
\usepackage[lite, initials]{amsrefs}
\newtheorem{teorema}{Theorem}[section]
\newtheorem*{theorem*}{Main Theorem}
\newtheorem{lemma}[teorema]{Lemma}
\newtheorem{propos}[teorema]{Proposition}
\newtheorem{corol}[teorema]{Corollary}
\theoremstyle{definition}

\newtheorem{rem}{Remark}[section]
\newtheorem{defin}[teorema]{Definition}

\def\R{{\mathbb R}}  
  
\def\C{{\mathbb C}}  
  
\def\Z{{\mathbb Z}}

\def\Ol{\mathcal{O}}

\def\Re{{\sf Re}}
\def\Crt{{\rm Crt}}

\def\oli{\overline}				
   
\def\p{\partial}

\title{Some properties of Grauert type surfaces}
\author[S.~Mongodi]{Samuele Mongodi\textsuperscript{1}}
\address[\textsuperscript{1}]{Dipartimento di Matematica - Universit\`a di Pisa, Largo Bruno Pontecorvo 5, I--56127, Pisa,  Italy}
\email{mongodi@dm.unipi.it} 
\thanks{The first author was supported by the FIRB2012 grant ``Differential Geometry and Geometric Function Theory''}
\author[Z.~Slodkowski]{Zbigniew Slodkowski\textsuperscript{2}}
\address[\textsuperscript{2}]{Department of Mathematics, University of Illinois at Chicago, 851 South Morgan Street, Chicago, Illinois 60607, Usa}
\email{zbigniew@uic.edu}
\author[G.~Tomassini]{Giuseppe Tomassini\textsuperscript{3}}
\address[\textsuperscript{3}]{Scuola Normale Superiore, Piazza dei Cavalieri, 7 - I-56126 Pisa, Italy}
\email{g.tomassini@sns.it}

  \date{\today}
 \subjclass[2010]{Primary 32E, 32T, 32U; Secondary 32E05, 32T35, 32U10}
\keywords{Pseudoconvex domains, Weakly complete spaces, Holomorphic foliations}

\begin{document}

\maketitle
\tableofcontents

\section{Introduction}

Let $X$ be a complex surface, with a real analytic plurisubharmonic exhaustion function $\alpha$, such that the regular level sets of $\alpha$ are Levi-flat hypersurfaces, foliated by dense complex leaves; following the classification of the Main Theorem in \cite{crass}, we call $X$ a surface \emph{of Grauert type}. 

It follows from \cite[Theorems 5.1 and 6.1]{mst} that, up to removing a compact complex curve or passing to a holomorphic double cover, there exists a proper pluirharmonic function $\chi$ such that $\alpha=\lambda\circ \chi$.

In this short note we want to analyze some properties of Grauert type surfaces and give an example of the case \emph{iii-b} of the Main Theorem in \cite{crass}, i.e. an instance where we need to pass to a double holomorphic cover to be able to produce the desired pluriharmonic function.

In Section 2, we study the compact complex curves which can be found in a Grauert type surface and we show that they are all negative curves in the sense of Grauert, sitting in the singular levels of $\alpha$.

Section 3 is devoted to prove that the level sets of the proper pluriharmonic function $\chi$ are connected, showing in some sense that such a function is ``minimal'' and completing the analogy with the Cartan-Remmert reduction of a holomorphically convex space: for a surface of Grauert type, up to a double cover, we produced a proper pluriharmonic function with connected level sets such that all the other pluriharmonic functions are obtained by composing it with the (pluri)harmonic functions of the image (i.e. linear functions, the image being $\R$).

In Section 4 we construct a family of examples of complex surfaces of Grauert type which do not admit a proper pluriharmonic function, but whose holomorphic double covers do.

\section{Compact curves in Grauert type surfaces}

We know from \cite[Theorem 4.2]{mst} that no regular level of $\chi$ can contain a compact curve, because this would imply that $X$ is a union of compact complex curves, which is not the case, $X$ being of Grauert type. This section's results inspect what happens in the singular levels, but first we need the analogue of \cite[Lemma 7.2]{mst} for singular curves.

\begin{lemma}\label{ottobre}Let $W$ be a complex $2$-dimensional manifold and $C$ a compact complex curve in $W$. Then there is a finite covering $\{B_j\}_{j=1}^n$ of $C$ by open subsets of $X$ such that
\begin{enumerate}[i)]
\item $C\subset \bigcup_{j=1}^n B_j$
\item every $B_j$ is simply connected
\item $B_j\cap C$ is connected for  $j=1,\ldots, n$
\item whenever $B_j\cap B_k\neq \emptyset$, then also $B_j\cap B_k\cap C\neq\emptyset$ and $B_j\cap B_k$ is connected, for $1\leq j,k\leq n$.
\end{enumerate}
\end{lemma}
\begin{proof} Since $C$ is a semianalytic subset of $W$, by \cite[Theorem 2]{Lo}, we find a pair $(K,L)$, with $K$ a locally finite simplicial complex and $L\subset K$ a finite closed subcomplex,  which gives a triangulation of the pair $(X,C)$, through a homeomorphism $h:(|K|,|L|)\to(X,C)$. Consider  a barycentric subdivision $K^1$ of $K$ and the induced subdivision $L^1$ of $L$.

All the simplices are considered to be open. We write $x\prec y$ if $x$ is a face of $y$.

For every $a\in L^1$, we define the \emph{star of $a$ with respect to $K^1$} as
$$B_a=\bigcup_{a\prec s\in K^1}s\;.$$
It is clear that $B_a$ gives, through $h$, an open, simply connected neighbourhood $U_a$ of $h(|a|)$ in $X$, which intersects $C$ in a connected set. $L^1$ being a finite subcomplex, its vertices are in finite number, hence we have a finite collection of open sets $U_a$ with $a$ a vertex of $L^1$.

If $B_a\cap B_b\neq\emptyset$, we find a simplex $s_1\in B_a\cap B_b$ and it is not difficult to show that its edge $(a,b)$ is contained in $L^1$. In terms of open sets $U_a$, $U_b$, we obtained that if $U_a\cap U_b\neq \emptyset$ then $U_a\cap U_b\cap C\neq\emptyset$ as well.

Finally, since $B_a\cap B_b$ is a union of simplices of $K^1$ and every such simplex has $(a,b)$ as a face, $B_a\cap B_b=B_{(a,b)}$, i.e. the star of $(a,b)$ in $K^1$; therefore it is connected. \end{proof}

Now, we are in the position to formulate and prove the counterpart of Lemma 4.1 in \cite{mst} for critical levels.

\begin{lemma}\label{Atropo}Let $\chi\!:W\to \R$ be a pluriharmonic function on a complex surface $W$; assume that $d$ is a critical value of $\chi$ and that the level set $\{\chi=d\}$ contains a connected compact complex curve $C$. Then 
\begin{enumerate}
\item there exist a neighborhood $V$ of $C$ and nonconstant holomorphic function $G:V\to\C$ such that $G$ vanishes on $C$;
\item if, in addition, $W$ is a Grauert type surface, we can construct $G$ such that $\{G=0\}$ is connected but different from $C$.
\end{enumerate}
\end{lemma}
\begin{proof} 
Choose a covering $\left\{V_j\right\}_{j=1}^n$ of $C$ as given by Lemma \ref{ottobre}. Consider $\chi_{_{\vert V_j}}$, $j$ fixed. Assuming $r>0$ small enough
so that $V_j$ is contained in a topological ball in $W$, we conclude
that $\chi$ has a pluriharmonic conjugate in this ball, and so in
$V_j$, say $\tau'_j\!:V_j\to\R$. Since $\chi_{_{\vert C}}=d$, a
constant, we conclude that its harmonic conjugate on $C\cap V_j$ is locally constant. Since $C\cap V_j$
is connected ${\tau'_j}_{\vert C\cap V_j}$ is constant. Subtracting the latter constant from
$\tau'_j$ we obtain a function $\tau_j\!:V_j\to\R$ such
that
\vspace{3mm}
\begin{itemize}
\item[$\bullet$] $\tau_j$ is a pluriharmonic conjugate of $\chi_{\vert
V_j}$;\\
\item[$\bullet$] $\tau_{\vert C\cap V_j}\equiv
0.$
\end{itemize}
Consider now two intersecting neighborhoods  $V_j$, $V_k$ and define
$V_0\!:=V_j\cap V_k\neq\emptyset.$ Since $V_0$ is connected, $\tau_j-\tau_k=a$
constant in $V_0$ and so $\tau_j-\tau_k\equiv 0$ in $V_0$,
because ${\tau_j}_{\vert V_0\cap C}\equiv 0\equiv
{\tau_k}_{\vert V_0\cap C}$. (Note that 
$V_0\cap\,C\neq\emptyset$ if $V_0\neq\emptyset$.) Thus $\tau_j(p)=\tau_k(p)$, whenever $p\in V_j\cap V_k$. Consequently, the family $\{\tau_j\}_{j=1}^n$ defines a single-valued pluriharmonic function $\tau\!:\!V\to\!\R$, where $V\!\!=\bigcup\limits_{j=1}^nV_j$, such that $F(p)\!=\!\chi(p)+i\tau(p)$, $p\in V$, is holomorphic.

Therefore, there exists $F\in\mathcal{O}(V)$, a non constant holomorphic function, such that $F_{\vert _C}=d$. We take $d=0$, proving (1).

\bigskip

To show (2), suppose that $W$ is a Grauert type surface and, by contradiction, that $C$ is a connected component of $\{F=0\}$. Take $V_0\subseteq V$ connected such that $F_0=F\vert_{V_0}:V_0\to\C$ is proper and $\{F_0=0\}=\{F=0\}\cap V_0=C$. Obviously, $\chi(V_0)$ is an open interval in $\R$ and it contains at least one regular value $t_0$; consider $p_0\in V_0$ such that $\chi(p_0)=t_0$ and define
$$C_0=\{p\in V_0\ :\ F_0(p)=F_0(p_0)\}\;.$$
By properness, this is a compact complex curve in $\chi^{-1}(t_0)$; but, $W$ being a Grauert type surface, regular levels of $\chi$ are foliated by dense (hence non compact) complex curves. This is a contradiction, so $C$ cannot be a connected component of $\{F=0\}$. \end{proof}
%
%
%
%
%

The previous result is enough for our purposes, however, we can prove a bit more about compact curves in Grauert type surfaces.

\medskip

We recall that a compact complex curve $C$ in a complex surface $X$ is said to be \emph{negative} if it has negative self-intersection or, equivalently, if $\Ol(-C)$ has a nonzero section.

It is a celebrated result of Grauert (see \cite{Gra}) that this is equivalent to being contractible, i.e. to the existence of  a proper holomorphic map $f:X\to X'$ onto a complex space $X'$, a biholomorphism outside $C$, such that $f(C)$ consists of just one point.

Therefore, negative curves arise as a result of modification of complex (possibly singular) surfaces. We state the next result for any $W$ not foliated in compact curves, which includes Grauert type surfaces and modification of Stein spaces, thus presenting also an alternative proof of part (2) of the previous Lemma.

\begin{corol}Let $W$, $\chi$, $d$, $C$ be as in Lemma \ref{Atropo}; if $W$ is not foliated in compact curves, then $C$ is negative. \end{corol}
\begin{proof} Let $\{U_j\}$ be an open convering of $C$ such that for each $j$ there exists a function $f_j\in\mathcal{O}(U_j)$ with $C\cap U_j=\{f_j=0\}$. The functions $f_i/f_j$, defined whenever $U_i\cap U_j\neq\emptyset$, represent the cocycle of the line bundle $\mathcal{O}(C)$; if such a bunlde were trivial, the function $f_j$ would glue into a global defining function for $C$, $G:U\to\C$, where $U=\bigcup_j U_j$. Hence, $C_\zeta=\{p\in U\ : \ G(p)=\zeta\}$ would be a compact curve, for $|\zeta|$ small enough; at least one of such curves would be contained in a regular level of $\chi$, so, by Theorem 4.2 in \cite{mst}, $W$ would be foliated in complex curves.

Therefore, $\Ol(C)$ is not trivial. Consider the holomorphic function $F:V\to\C$ given by Lemma \ref{Atropo}-(1) and the collection of functions $\{F/f_j\}_j$, each defined on $V\cap U_j$; such collection defines a nonzero holomorphic section $\sigma$ of the line bundle $\mathcal{O}(-C)\cong\mathcal{O}(C)^*$. As this line bundle is not trivial, $\sigma$ has to vanish at some point of $C$. This implies that $C\cdot C<0$, i.e. $C$ is negative in the sense of Grauert. \end{proof}

In particular, any compact curve in a Grauert type surface, apart from the absolute minimum set of the real analytic plurisubharmonic exhaustion function, were this of real dimension $\leq 2$, is negative.

\section{Level sets of pluriharmonic functions}\label{level}

\begin{teorema}\label{lisu}
Let $W$ be a $2$-dimensional complex manifold and $\chi:W\to (a,b)$, $-\infty\le a<b\le+\infty$, a proper pluriharmonic function. Assume $W$ is a Grauert type surface. Then every level of $\chi$ is connected.
\end{teorema}
We need first some auxiliary facts about convergent sequences of analytic sets. For a discussion of the various definitions of convergence, see \cite[Section 15.5]{chi} or \cite{two}. It is convenient for us to reformulate these definitions in the following manner.

Let $Y$ be a locally compact Hausdorff space. Let $\{F_n\}$ be a sequence of non-empty closed subsets of $Y$. We say that 
$\limsup_{n\to+\infty} F_n\subset F$, where $F$ is a closed subset of $Y$, if for every compact subset $K\subset Y\setminus F$ there is an index $n_0$ such that for $n>n_0$, $F_n\cap K=\emptyset$. We say that $\limsup_{n\to+\infty} F_n=F_0$, if 
$$
F_0=\bigcap\limits_{\stackrel{\limsup_{n\to+\infty} F_n\subset F}{F\,{\rm closed}}}  F.
$$
Then, using \cite[Lemma 1(1)]{two} we can define that the sequence $\{F_n\}$ is convergent to $F_0$ ($\lim_{n\to+\infty}F_n=F_0$) if 
\begin{itemize}
\item[i)] $\limsup_{n\to+\infty} F_n=F_0$;
\item[ii)] for every $y_0\in F_0$ there is a sequence $\{y_n\}$ such that $y_n\in F_n$, $n=1,2,\ldots$ and $\lim_{n\to+\infty}y_n=y_0.$
\end{itemize} 
Observe that $\limsup_{n\to+\infty} F_n$ always exists $\lim_{n\to+\infty}F_n$ not always.

Although we believe the following proposition to be well known, we will give a sketch of proof.
\begin{propos}\label{lisup}
Let $W$ be a complex manifold and $\{F_n\}$ a sequence of closed sets with local maximum property. Then $F:=\limsup_{n\to+\infty} F_n$ has the local maximum property as well.
\end{propos}
\begin{proof} (Sketch) Suppose, by contradiction, that $F$ does not have the local maximum property, then, by \cite[Proposition 2.3]{Sl2}, there are $y\in F$, $B=B(y,r)$ with $r>0$, an $\epsilon>0$ and a strongly plurisubharmonic function on $\overline{B}$ such that $u(y)=0$ and
\begin{equation}\label{Mercuzio}u(z)<-\epsilon|z-y|^2\;,\qquad z\in\overline{B}\cap F\setminus\{y\}\;.\end{equation}
By  the very definition of limit superior of analytic sets we have
$$\limsup_{n\to\infty}(F_n\cap bB)\subseteq F\cap bB\;;$$
moreover, since $u$ is continuous, we have that
\begin{equation}\label{Tebaldo}\limsup_{n\to\infty}\max u\vert_{F_n\cap bB}\leq \max u\vert_{F\cap bB}\;.\end{equation}
Now, it follows again from the definition of limit superior of analytic sets that there is a subsequence $F_{n_k}$ and a sequence of points $(y_k)$ such that $y_k\in F_{n_k}$ and $\lim y_k=y$, therefore $\lim u(y_k)=u(y)=0$. Hence, there exists $k_0$ such that
$$u(y_{k_0})>\max u\vert_{F_{n_{k_0}}\cap bB}$$
by \eqref{Mercuzio} and \eqref{Tebaldo}; so
$$\max u\vert_{F_{n_{k_0}}\cap \overline{B}}>\max u\vert_{F_{n_{k_0}}\cap bB}\;,$$
which contradicts the local maximum property of $F_{n_{k_0}}$. \end{proof}

\begin{corol}\label{lisup1}
Let $W$ be a complex manifold and $\{F_n\}$ a sequence of closed complex analytic subsets of $W$ of pure dimension $1$. Assume that $F_0:=\limsup_{n\to+\infty} F_n$ is contained in a complex analytic subset $F$ of $W$, $F$ also of pure dimension $1$. Then $F_0$ is the union of some of the irreducible components of $F$. 
\end{corol}
\begin{proof}
(Sketch) Let $V$ be a connected component of $F_{\rm reg}$. $V$ is a Riemann surface. If $F_0\cap V\neq\emptyset$, and $F_0\cap V\neq V$, i.e. has a boundary point $x_0$ in $V$, then $F_0$ fails to have local maximum property. Thus $F_0\cap F_{\rm reg}$ is the union of a family of connected components $V_\alpha$ of $F_{\rm reg}$, i.e. $F_0=\cup_\alpha\oli V_\alpha$, where each $\oli V_\alpha$ is an irreducible variety. \end{proof}

 The following fact is a special case of a much more general result in \cite[Theorem 1]{two}. 
 
\begin{corol}\label{lisup2}
Let $W$ be a complex manifold of dimension $2$. Let $Z_1$ and $Z_2$ be (closed) complex analytic subsets of $W$, of pure dimension  $1$. Assume that $Z_1\cap Z_2\neq\emptyset$ is discrete. Now let $\{F_n\}$ be a sequence of complex analytic sets of pure dimension $1$ converging to $Z_2$. Then $Z_1\cap F_n\neq\emptyset$ for $n$ large enough.
\end{corol}
\begin{rem}
The corresponding result in \cite{two} is stated only in $\C^n$ requiring $Z_1\cap Z_2$ to be finite, but since it is essentially local, it is true for complex manifolds and with $Z_1\cap Z_2$ discrete without need to change the proof.
\end{rem}
\begin{teorema}\label{lisup3}
Let $W$ be a complex manifold of dimension $2$ and $\{F_n\}$ a convergent sequence of irreducible complex analytic subsets of $W$ of pure dimension $1$. Let $Z$ be a connected complex analytic subset of $W$ with $\dim_\C Z=1$. Assume that 
\begin{itemize}
\item[i)] $F_n\cap Z=\emptyset$, $n=1, 2, \ldots$ 
\item[ii)] $\lim_{n\to+\infty} F_n\subseteq Z$.
\end{itemize}
Then $Z=\lim_{n\to+\infty}F_n$.
\end{teorema}

Of course the whole point of this result is that it holds when $Z$ is reducible (otherwise it is a trivial consequence of Corollary \ref{lisup1}).

\begin{proof}
Denote $Z_2=\lim_{n\to+\infty}F_n.$ Suppose that $Z_2\neq Z$. By Corollary \ref{lisup1}, $Z_2$ is the union of some family of irreducible components of $Z$. Denote by $Z_1$ the union of the remaining irreducible components of $Z$. As it is well known, the family of all irreducible components being locally finite, $Z_1$, $Z_2$ are closed, $Z_1\cap Z_2$ discrete and closed. Since $Z$ is connected, $Z_1\cap Z_2$ is nonempty.

Applying Corollary \ref{lisup2}, we deduce that $Z_1\cap F_n\neq \emptyset$ for $n$ large enough, but this contradicts assumption \emph{ii)}, unless $Z_1=\emptyset$ and $Z=\lim_{n\to\infty} F_n$.\end{proof}

%

\begin{rem}
It is clear that the assumption of irreducibility or even connectdness of  $F_n$ is not needed, and was actually not used in the proof. 
\end{rem}
\begin{lemma}\label{lisup4}
Let $W$ be a complex manifold of dimension $2$ and $D$ a domain in $W$. Let $f:D\to\C$ be a holomorphic function. Let $C$ be a (possibly reducible) connected compact complex curve in $D$, such that $f=0$ on $C$, and that the subset $\{x\in D:f(x)=0\}$ is connected and different from $C$. Then $C$ cannot separate any of its relative neighborhoods in the level set $\{\chi=0\}$ where $\chi=\Re f$.
\end{lemma}
\begin{proof}
Consider any connected neighborhood $U$ of $C$ in $W$, such that $U\cap\{\chi=0\}$ is connected. By diminishing $U$ (if needed), we can assume also that the variety $Z:=\{f=0\}$ is connected. It is crucial for the proof that $Z\setminus C$ is nonempty. Suppose that $(U\cap\{\chi=0\})\setminus C$ is disconnected, and that $(A_j)$ is the family (finite or countable) of its pairwise disjoint relatively open connected components, i.e. $\cup_{j=1}^\infty A_j=(U\cap\{\chi=0\})\setminus C.$ Observe that $\oli A_j\cap U\subseteq A_j\cup C$  and $\oli A_j\cap U\cap C\neq\emptyset$ for every $j$.

Consider now any set $A_j$. Denote $\lambda={\sf Im} f$. Observe that $\lambda$ cannot be constant on $A_j$, because then $\chi_{|A_j}=\lambda_{|A_j}=0$ and consequently $A_j\subset \{f=0\}$ is not relatively open in $\{\chi=0\}$. Thus $\lambda(A_j)$ is an interval, either $(0,\epsilon_0)$ or $(-\delta_0,0)$, $\epsilon_0,\delta_0>0$. Let $\{t_n\}$ be a sequence of values belonging to $\lambda(A_j)$, such that $\lim_{n\to+\infty}t_n=0$. The sets 
$$
\{\lambda=t_n\}\cap A_j=\{f=it_n\}\cap A_j
$$  
are nonempty for every $n$. Denote by $F_n$ any of the (irreducible) connected components of this set. Since $f$ is continuous, $\limsup_{n\to+\infty} F_n\subset Z$. By the properties of the topology of (locally uniform convergence) of closed sets defined above, $F_n$ has a convergent subsequence, which we still denote by $\{F_n\}$. By Theorem \ref{lisup3} 
$\lim_{n\to+\infty} F_n=Z$ and so $Z\subset\oli A_j\cap U$. Moreover, $\oli A_j\cap U\subseteq A_j\cup C$. It follows that $\emptyset\neq Z\setminus C\subset A_j$, and so $\bigcap_jA_j$ is nonempty. This is a contradiction. 
\end{proof}

\begin{lemma} \label{sconnessione}Let $W, B$ be normal, Hausdorff topological spaces and $f:W\to B$ a continuous, proper, surjective map such that
\\

$(\star)$\ \ \parbox{10cm}{ for every $b\in B$, for every $Y$ connected component of $f^{-1}(b)$ there is $U$ a neighbourhood of $b$ such that  if $X_0$ is the connected component of $f^{-1}(U)$ containing $Y$ then for every $b'\in U$, $f^{-1}(b')\cap X_0$ is connected and non empty.}
\\

\noindent If $B$ is simpy connected, and $W$ is connected, then $f^{-1}(b)$ is connected for every $b\in B$.
\end{lemma}
\begin{proof} Define on $W$  the following equivalence relation: given $u,w\in W$, we write $u\sim w$ if $f(u)=f(w)$ and both lay in the same connected component of $f^{-1}(f(u))$. Let $\pi:W\to W/\sim$ be the quotient map. It is easy to check that the map induced by $f$, namely $\bar{f}:W/\sim\to B$, is a covering map by $(\star)$, together with properness of $f$.

As $B$ is simply connected, the covering spaces of $B$ can be only disjoint unions of copies of $B$, but if $W$ is connected, so is $W/\sim$. Therefore $\bar{f}$ is an homeomorphism, hence $f^{-1}(b)$ is connected for every $b\in B$ (and non empty by surjectivity). \end{proof}

We note that a locally trivial fibration satisfies condition $(\star)$.

\begin{teorema}\label{affettato} Let $X$ be a connected smooth manifold, $f:X\to \R$ a proper smooth surjective function such that
\begin{enumerate}
\item the critical values of $f$ are discrete
\item every connected component of $\Crt(f)$ does not disconnect any of its open neighbourhoods in $X$
\item for any $t\in\R$ and any connected component $Y$ of $f^{-1}(t)$, $Y\setminus\Crt(f)$ is connected and non empty.
\end{enumerate}
Then $f^{-1}(t)$ is connected and non empty for every $t\in\R$.
\end{teorema}
\begin{proof}
We will prove that $f$ enjoys property $(\star)$ from Lemma \ref{sconnessione}.

Consider $t_0\in \R$ and let $Y$ be a connected component of $f^{-1}(t_0)$. There exists $\delta>0$ such that no critical values of $f$ are contained in $(t_0-\delta,t_0)$ or $(t_0,t_0+\delta)$. Let $X_0$ be the connected component of
$$\{x\in X\ :\ t_0-\delta<f(x)<t_0+\delta\}$$
containing $Y$.

\medskip

Up to taking $\delta$ small enough, If $Y\cap \Crt(f)=\emptyset$, then $X_0\cap \Crt(f)=0$ and, by Ehresmann theorem (see Theorem 9.3 and Remark 9.4 in \cite{voi}), all the level sets are diffeomorphic. In particular, $Y$ being connected, every level set of $f\vert_{X_0}$ is connected; moreover, as $Y$ does not contain critical points, it cannot be neither a maximum nor a minimum set for $f$, hence there is $\delta'>0$ such that $f(X_0)\supseteq(t_0-\delta',t_0+\delta')$.

\medskip

If $Y\cap \Crt(f)\neq\emptyset$, consider $Y'=Y\setminus \Crt(f)$ as a (connected) submanifold of $X_0'=X_0\setminus\Crt(f)$, which is also connected. Take $U$ to be a tubular neighbourhood of $Y'$ in $X_0'$ and set
$$U^+=U\cap\{f>t_0\}\qquad U^-=U\cap\{f<t_0\}\;.$$
By definition, $Y'$ is disjoint from $\Crt(f)$, so it cannot contain extremal points for $f$, therefore $U^+$ and $U^-$ are both non empty and disjoint; moreover, $U$ is diffeomorphic to the total space of a (real) line bundle over $Y'$, hence $U\setminus Y'=U^+\cup U^-$ has at most $2$ connected components, so $U^+$ and $U^-$ are connected.

Obviously, $U^\pm\subseteq X_0^\pm=X_0\cap\{f\gtrless t_0\}$; let $V$ be a connected component of $X_0^+$.

As $X_0'$ is connected and $V\neq X_0'$,  $bV\cap X_0'$ is nonempty. Since $V$ is relatively closed in $X_0'$, $bV\cap X_0'\subseteq bX_0^+\cap X_0'=Y'$.
Thus $U\cap V$ is non empty and so is $U^+\cap V$ (because $V\cap U^+=V\cap U\cap X_0^+=V\cap U$); as this holds for every connected component of $X_0^+$ and as $U^+$ is connected, we conclude that $X_0^+$ is connected.

By Ehresmann theorem, $f\vert_{X_0^+}$ is a locally trivial fibration, hence satisfies condition $(\star)$ and $f(X_0)$ is a connected subset of $\R$, i.e. an interval, which is simply connected. By Lemma \ref{sconnessione}, all the level sets are connected.

The same holds for $X_0^-$ and hence for $X_0$. Moreover, it is clear that there exists $\delta''>0$ such that $f(X_0)\supseteq (t_0-\delta'',t_0+\delta'')$.

\medskip

Hence, $f$ enjoys property $(\star)$. Applying again Lemma \ref{sconnessione}, we conclude that every level set of $f$ is connected and, by surjectivity, non empty.
\end{proof}

{\bf Proof of Theorem \ref{lisu}} Let $C$ be any connected component of the critical set of $\chi$. 

The critical set is locally defined by holomorphic equations $\p\chi/\p z=0$, $\p\chi/\p z=0$, where $z,w$ are local coordinates. It follows that $C$ is a complex curve (possibly reducible). Moreover, since $\chi_{|C}=c$, $c$ a constant, and $\chi$ is proper, $C$ is compact complex curve or a point. 

Claim: $C$ cannot separate any of its relative neighborhoods in the level $\chi=c$.


As $X$ is of Grauert type, by Lemma \ref{Atropo}-(2) there exist a neighbourhood $U$ of $C$ in $X$ and a holomorphic function $F:U\to\C$ such that $C\subsetneq \{F=0\}$. Our claim follows from Lemma \ref{lisup4}.

In case $C$ is just a point $p$, the proof is similar, just simpler. We can take for $U$ a simply-connected neighborhood of $p$ and $\lambda$ the conjugate of $\chi{|U}$, $f=\chi+i\lambda-c$ in $U$. Shrink $U$ so that $Z=\{f=0\}\cap U$ is connected. Clearly, $Z\neq C$. Inspection of the proof of Lemma \ref{lisup4} shows that it holds for $C$- a single point- without any change. The claim is proved. 

As $\chi$ satistfies the hypotheses of Theorem \ref{affettato}, we conclude that the level sets of $\chi$ are connected.  \hfill $\Box$

\medskip

The following Corollary is obvious.

\begin{corol} Grauert type surfaces have at most two ends.\end{corol}

\section{A class of examples}

The purpose of this section is to provide examples to show that passing to a double cover  in the case \emph{iii-b} of the Main Theorem proved in \cite{mst} is unavoidable. The result can summarized as follows.
\begin{teorema}\label{teo1}
There exists a weakly complete complex surface $X^\ast$ such that
\begin{enumerate}
\item $X^\ast$ admits a real analytic plurisubharmonic exhaustion function and is a Grauert type surface
\item $X^\ast$ admits a (holomorphic) involution $T$ without fixed points
\item $X=X^\ast/\langle T\rangle$ is again a Grauert type surface with a real analytic plurisubharmonic exhaustion function
\item every proper pluriharmonic function on $X$ is constant.
\end{enumerate}
\end{teorema}
%
 The surface $X^\ast$ discussed in Theorem \ref{teo1} will be obtained by modification of \cite[Example 2.3]{mst}. In rough outline, consider compact Riemann surfaces $M$ which admit a fixed-point free conformal involution $S$, and $X^\ast$ the space of some, topologically trivial, line bundle on $M$ without the $0$ section, where the line bundle transformed by $S$ is equivalent to the dual line bundle. There are more details to be taken care of, as discussed next.
 
\begin{rem} We will show, in fact, that there exist infinitely many such surfaces.\end{rem} 
%
%

Before starting with our construction, we recall a rather elementary result.

\begin{lemma}\label{lemma_odd}Let $g>0$ be an odd integer. Then there exists a compact Riemann surface $M_g$ of genus $g$ with a conformal involution $S:M_g\to M_g$ without fixed points.\end{lemma}
\begin{proof} We have $g=2g'-1$ for some positive integer $g'$. Let $\Sigma$ be a compact Riemann surface of genus $g'$ and let $\pi:V\to \Sigma$ be a (topological) double cover of $\Sigma$. 
Such a double cover exists because, for example, the abelianization of $\pi_1(\Sigma)$, which is $H^1(\Sigma, \Z)=\Z^{2g'}$, admits a surjection on $\Z/2\Z$. Therefore the fundamental group of $\Sigma$ has a (normal) subgroup of index $2$.

Now, we can pull back the complex structure from $\Sigma$ to $V$, giving it the structure of a Riemann surface, so that $\pi$ becomes a holomorphic map; by the Riemann-Hurwitz formula for an unramified cover of degree $2$,  $\chi(V)=2\chi(\Sigma)$ and, remembering that $\chi(M_g)=2-2g$, we immediately obtain that the genus of $V$ is $2g'-1$.

Then $V$ is  our $M_g$ and we have that there exists $S\in \textrm{Aut}(M_g)$ such that $\Sigma= M_g/\{I,S\}$. Hence $S$ is the involution we were looking for.
\end{proof}

\subsection{$S$-antisymmetric bundles and cocycles}

 We recall now some basic definitions and notation.
 
 If $S$ is an automorphism of a Riemann surface $M$ and $\xi$ a complex line bundle with cocycle $\{\xi_{ij}\}_{i,j\in A}$, $\xi_{ij}:U_i\cap U_j\to\C$, where $\{U_i\}_{i\in A}$ is an open covering of $M$, we denote by $S(\xi)$ the $S$-image (i.e. the pushforward by $S$) of $\xi$, by which we understand the line bundle with the defining cocycle $\{k_{ij}\}_{i,j\in A}$, where $k_{ij}:S^{-1}(U_i)\cap S^{-1}(U_j)\to\C$ is $k_{ij}:=\xi_{ij}\circ S$, corresponds to the covering $\{S^{-1}(U_i)\}_{i\in A}$.
 
 From now on we assume $S$ is an involution without fixed points, i.e. $S\circ S={\rm Id}$ and $S(m)\neq m$ for all $m\in M$.
 
 Crucial to our construction are $S-antisymmetric$ line bundles i.e, such that $S_\ast(\xi)=\xi^{-1}$ ( ``='' means line bundle equivalence; more precise requirements will be indicated in terms of cocycles). Since $S$-image preserves Chern numbers, while $\xi\mapsto \xi^{-1}$ reverses them, $S$-antisymmetric bundles have Chern number $0$, and so are topologically trivial. As well known, they can be represented by flat unimodular cocycles $\{\xi_{i,j}\}$, i.e. $\xi_{i,j}=\{\rm const\}$ in
$U_i\cap U_j$, $\vert \xi_{i,j}\vert=1$ (and, of course, $\xi_{ij}\xi_{jk}=\xi_{ik}$).

Our surface $X^\ast$ will be the space of an $S$-anntisymmetric line bundle $\eta$ over $M$ minus its $0$ section, where $S$ is fixed-points free but in order to define on $X^\ast$ a biholomorphic involution (determined by $S$), we need to use an $S$-antisymmetric cocycle defined next.
\begin{defin}\label{def3}
With $M$ and $S$ as above consider a complex line bundle cocycle $\{\eta_{ij}\}_{i,j\in A}$ associated with an open covering $\{V_i\}_{i\in A}$ of $M$, where $\eta_{ij}:V_i\cap V_j\to\C$. We say that $\{\eta_{ij}\}_{i,j\in A}$ is an $S$-antisymmetric cocycle if there is an involutive permutation $\sigma$ of the index set $A$, i.e. $\sigma:A\to A$, $\sigma\circ\sigma={\sf Id}_A$, such that
\begin{itemize} 
\item[a)] $S^{-1}(V_i)=V_{\sigma(i)}$, $i\in A$
\item[b)] $\eta_{\sigma(i)\sigma(j)}=1/\eta_{ij}\circ S^{-1}$ in $V_{\sigma(i)}\cap V_{\sigma(j)}=S^{-1}(V_i\cap V_j)$.
\end{itemize}
\end{defin} 

Another useful concept will be that of $S$-stable covering.
\begin{defin} A covering $\{U_j\}_{j=1,\ldots, 2n}$ of $M$ is called $S$-stable if $S(U_j)=U_{j+n}$ for $j=1,\ldots, 2n$, where the indices are taken $\mod 2n$.
\end{defin}

We notice that, given any covering $\{W_j\}_{j\in A}$ of $M$, we can produce an $S$-stable covering subordinated to that: for every $p\in M$, let $U_p$ be a neighbourhood of $p$ such that
\begin{itemize}
\item $U_p\cap S(U_p)=\emptyset$
\item $\exists h,k\in A$ such that $U_p\subseteq W_k$, $S(U_p)\subseteq W_h$.
\end{itemize}
Now, take $p_1,\ldots, p_n$ such that $\{U_{p_j}\}_j \cup \{S(U_{p_j})\}_j$ is a covering of $M$ and reindex such covering so that $S(U_j)=U_{j+n}$ for $j=1,\ldots, n$.

\begin{lemma}\label{lem5}Let $\xi$ be an $S$-antisymmetric bundle. Then  $\xi^{\otimes2}$ admits an $S$-antisymmetric unimodular representing cocycle with respect to an $S$-stable covering.\end{lemma}
\begin{proof} First of all, we note that also $\xi^{\otimes2}$ is an $S$-antisymmetric bundle.

We have that $S(\xi)=\xi^{-1}$, so there exists a covering $\{W_j\}_{j\in A}$ with respect to which the cocycles of $S(\xi)$ and $\xi^{-1}$ are equivalent.

 We can refine the covering to an $S$-stable one $\{\widetilde{W}_j\}_{j=1,\ldots, n}$, so that, if $\{\eta_{ij}\}_{i,j=1,\ldots, 2n}$ is the cocycle representing $\xi^{-1}$, then $\{\eta_{i+n, j+n}^{-1}\circ S\}_{i,j=1,\ldots, 2n}$ is the cocycle representing $S(\xi)$.
 
 We suppose that $\widetilde{W}_j$ and $\widetilde{W}_i\cap\widetilde{W}_j$ are simply connected for every $i,j$.
 
 Therefore, there exists a $0$-cochain $\{\phi_j\}$, $\phi_j\in\Ol^*(\widetilde{W}_j)$,  such that
 \begin{equation}\label{Roma}\phi_i\eta_{ij}\phi_j^{-1}=\eta_{i+n, j+n}^{-1}\circ S\end{equation}
 on $\widetilde{W}_{i}\cap\widetilde{W}_j$. Note that we also have
 $$(\phi_{i+n}\circ S)( \eta_{i+n,j+n}\circ S)( \phi_{j+n}^{-1}\circ S)=\eta_{ij}^{-1}\;.$$
 Dividing the last two equations sideways and cancelling the $\eta$-terms, we obtain
 $$\dfrac{\phi_i}{\phi_{i+n}\circ S}=\dfrac{\phi_j}{\phi_{j+n}\circ S}\qquad \textrm{ in } \widetilde{W}_i\cap\widetilde{W}_j\;.$$
 Thus, these ratios define a non-vanishing holomorphic function on $M$. Hence, there is a nonzero complex number $c\in\C^*$ such that
 \begin{equation}\label{Milano}\phi_i=c(\phi_{i+n}\circ S)\end{equation}
 for every $i=1,\ldots, 2n$.

 
\medskip

Let now $\nu_{ij}=\phi_i\eta_{ij}^2\phi_j^{-1}$; this is a cocycle representing $\xi^{\otimes 2}$. Moreover:
$$\begin{array}{rcll}\nu_{i+n, j+n}^{-1}\circ S&=&(\phi_{i+n}\circ S)^{-1}(\eta_{i+n, j+n}\circ S)^{-2}(\phi_{j+n}\circ S)&\textrm{  by definition}\\
&=&(\phi_{i+n}\circ S)^{-1}(\phi_i\eta_{ij}\phi_j^{-1})^2(\phi_{j+n}\circ S)&\textrm{  by \eqref{Roma}}\\
&=&(c^{-1}\phi_i)^{-1}\phi_i^2\eta_{ij}^2\phi_j^{-2}(c^{-1}\phi_j)&\textrm{ by \eqref{Milano}}\\
&=&\phi_i\eta_{ij}^2\phi_j^{-1}& \\
&=&\nu_{ij}&\end{array}$$
 
Therefore,  we have an $S$-antisymmetric cocycle $\{\nu_{ij}\}_{i,j=1\ldots 2n}$ with respect to the covering $\{\widetilde{W}_{j}\}_{j=1,\ldots, 2n}$.
It is easy to see that, if $\eta_{ij}$ was taken to be unimodular, so is $\xi_{ij}$.
\end{proof}

\begin{rem} We note that, if we apply twice \eqref{Milano}, we obtain that $c^2=1$, i.e. $c=\pm1$. Moreover, this number depends only on the cohomology class of $\xi$; if $c=1$, we could already produce an $S$-antisymmetric representing cocycle for $\xi$.\end{rem}

\begin{propos}\label{pro2}If $M$ is a compact Riemann surface of genus $g>1$, with a fixed-point free conformal involution $S$, then there exists an $S$-antisymmetric complex holomorphic line bundle $\eta$ which is not unipotent, i.e. $\eta^{\otimes k}$ is not holomorphically trivial for any $k\in\Z\setminus\{0\}$.
\end{propos}
\begin{proof}
The proof uses some basic facts about Picard variety, $\mathcal P(M)$, of the Riemann surface, for which we refer to Gunning \cite[\S 8]{GU}. $\mathcal P(M)$ has several descriptions; in the one crucial for us it is the group  of (holomorphic equivalence classes of) complex line bundles of Chern class $0$ (i.e. topologically trivial) on $M$. With the operation of tensor product $\mathcal P(M)$ becomes a group. For a positive genus $g$ of $M$, $\mathcal P(M)$ can be given a complex structure with which it becomes a complex Lie group, specifically a complex torus of complex dimension $g$, \cite[\S 8, p.146]{GU}. 
%

Consider the map $\Psi:\mathcal{P}(M)\to\mathcal{P}(M)$ defined by $\Psi(\xi):=\xi S(\xi)$, $\xi\in\mathcal{P}(M)$. It is clear that $\Psi$ is a homomorphism of complex Lie groups. The image of $\Psi$ is contained in the subgroup of line bundles $\eta$ such that $\eta=S(\eta)$; these line bundles are pullbacks of line bundles on $M/\{I,S\}$, so their subgroup is a complex torus of dimension $(g+1)/2$.

Consider $\mathcal{K}=\mathrm{ker}\ \Psi$; this is a Lie subgroup of $\mathcal{P}(M)$, and it cannot be countable (or finite) as $\mathcal{P}(M)/\mathcal{K}=\Psi(\mathcal{P}(M))$ is contained in a complex torus of dimension $(g+1)/2<g$, as $g>1$.
%
%

Recall now that, as a real Lie group, $\mathcal P(M)$ is isomorphic to the real torus $\big(\R/\Z\big)^{2g}$. The unipotent line bundles correspond to the points of real torus with all coordinates rational, and so form a countable set.
Since $\mathcal K$ cannot be countable, it must contain a non-unipotent bundle $\eta$ which proves the proposition.
\end{proof}

\begin{rem} The assumption $g>1$ is necessary. If $M$ is a torus, $\mathcal{K}$ is a discrete subgroup of $\mathcal{P}(M)$, which is compact, so $\mathcal{K}$ is finite and all its elements have finite order, i.e. are unipotent line bundles. Moreover, given the rather easy structure of the automorphisms of a torus and the isomorphism between a torus and its Jacobian, one can explicity compute the $S$-antisymmetric line bundles for any given involution $S$. \end{rem}

\subsection{Proof of Theorem \ref{teo1}}
Following Lemma \ref{lemma_odd}, let $M$ be a compact Riemann surface of odd genus $g>1$, with a conformal involution $S:M\to M$ without fixed points.

By Proposition \ref{pro2}, we have a non-unipotent $S$-antisymmetric line bundle $\eta$ on $M$, whose square $\nu=\xi^{\otimes2}$, by Lemma \ref{lem5}, can be represented with a unimodular $S$-antisymmetric cocycle $\{\nu_{ij}\}_{i,j=1,\ldots, 2n}$, with respect to an $S$-stable covering $\{U_j\}_{j=1,\ldots, 2n}$.

Let $X^\ast$ be the total space of $\nu$ minus the zero section $M$. Using local trivializations on the $S$-stable covering $\{U_j\}_{j=1,\ldots, 2n}$, we will construct a holomorphic involution $T:X^\ast\to X^\ast$ and a proper pluriharmonic function $\chi^\ast:X^\ast\to \R$.
%

Let $p:X^\ast\to M$ be the restriction to $X^\ast$ of the projection map of the bundle $\nu$; $p^{-1}(U_j)\cong U_j\times \C^*$, so that $X^\ast$ can be given as the union of charts $\{U_j\times \C^*\}_{j=1,\ldots, 2n}$, glued together with the following equivalence relation: if $U_i\cap U_j\neq\emptyset$, $(p,w_1)\in U_i\times\C^*,\ (q,w_2)\in U_j\times \C^*$ then
$$(p,w_1)\sim(q,w_2)\  \Leftrightarrow\ p\equiv q\textrm{ and }w_1=\nu_{ij}w_2$$

Define on each chart a pluriharmonic function $\chi_i:U_i\times\C^\ast\to\R$ by $\chi_i(p,w)=\log\vert w\vert$. Since $\vert \nu_{ij}\vert=1$ in $U_i\cap U_j$, ${\chi_i}_{\vert_{U_i\cap U_j}}={\chi_j}_{\vert_{U_i\cap U_j}}$ and so they define a single-valued pluriharmonic function $\chi^\ast:X^\ast\to\R$; it is easy to verify that $\chi^\ast$ is proper.

To define a biholomorphic involution $T$ of $X^\ast$, define first maps of charts 
$$
T_i:U_i\times\C^\ast\to U_{i+n}\times\C^\ast
$$ 
by
$$
T_i(p,w)=(S(p),1/w).
$$
Since $S(U_i)=U_{i+n}$ it is clear that $T_i$ is well defined (here and in what follows, the indexes are to be understood mod $2n$), and that 
$$
T_i(U_i\times\C^\ast)=U_{i+n}\times\C^*;
$$
$$
T_{i+n}\circ T_i={\rm Id}_{U_i}\qquad i=1,\ldots, 2n.
$$
It remains to check that whenever charts $U_i\times\C^\ast$ and $U_j\times\C^\ast$ overlap, $T_i$ and $T_j$ are equal on $\big(U_i\cap U_j\big)\times\C^\ast$.

Let $(p,w_1)\in U_i\times\C^\ast$, $(q,w_2)\in U_j\times\C^\ast$ and $(p,w_1)\sim (q,w_2)$ i.e. $p\equiv q$ and $w_1=\nu_{ij}(p)w_2$. Then also  $(S(p), 1/w_1)\sim (S(q), 1/w_2)$; indeed, $S(p)\equiv S(q)$, because $p\equiv q$, and $\nu_{ij}^{-1}=\nu_{i+n, j+n}\circ S$, so
$$\frac{1}{w_1}=\frac{1}{w_2}\nu_{ij}^{-1}(p)=\frac{1}{w_2}\nu_{i+n, j+n}(S(p))\;.$$
Therefore $T_i(p,w_1)\sim T_j(q,w_2)$. It follows that the family $\{T_i\}$ defines one biholomorphic involution $T:X^\ast\to X^\ast$.

It is clear from the definition that $T$ does not have fixed points if $S$ does not. 

We verify now that $\chi^\ast\circ T=-\chi^\ast$.

If $(p,w)\in U_i\times\C^\ast$, $T(p,w)=(S(p),1/w)$ and
\begin{eqnarray*}
\chi^\ast\circ T(p,w)&=&\chi^\ast(S(p),1/w)=\\
&&-\log\vert w\vert=-\chi^\ast(p,w).
\end{eqnarray*}

We define $\alpha^\ast(x)=|\chi^\ast(x)|^2$. This function is clearly plurisubharmonic and exhaustive, being proper and bounded from below.

As we know from \cite[Example 2.2]{mst}, if $\nu$ is not unipotent, $X^\ast$ is a Grauert type surface (in particular each complex leaf is dense in the level set of $\chi^\ast$ which contain it).

So far, we have proved points (1) and (2). The group $\langle T\rangle=\{I, T\}$ acting freely and properly discontinuously on $X^\ast$, we define the quotient 
$$\pi:X^\ast \to X=X^\ast/\langle T\rangle$$ 
which is a holomorphic covering map between complex surfaces. The function $\alpha^\ast$ descends as a real analytic plurisubharmonic exhaustion function $\alpha$ on $X$; the levels of $\alpha$ are still Levi flat hypersurfaces with dense leaves, hence $X$ is a Grauert type surface, showing (3).

Now, consider $\chi:X\to\R$ a pluriharmonic function; we can lift it to $\chi\circ\pi:X^\ast\to\R$, again pluriharmonic, so that $\chi\circ\pi= a\chi^\ast+b$ for some constants $a,b\in\R$ by \cite[Lemma 5.4 (iii)]{mst}. Since $\pi\circ T=\pi$, we have that
$$a\chi^\ast + b=\chi\circ \pi=\chi\circ\pi\circ T=a\chi^\ast\circ T + b=-a\chi^\ast +b$$
so $a=0$ and $\chi$ is constant. This proves (4). \hfill $\Box$


\begin{bibdiv}
\begin{biblist}
%
\bib{chi}{book}{
   author={Chirka, E. M.},
   title={Complex analytic sets},
   series={Mathematics and its Applications (Soviet Series)},
   volume={46},
   note={Translated from the Russian by R. A. M. Hoksbergen},
   publisher={Kluwer Academic Publishers Group, Dordrecht},
   date={1989},
   pages={xx+372},
   isbn={0-7923-0234-6},
   doi={10.1007/978-94-009-2366-9},
}

%
%
%
%

\bib{Gra}{article}{
   author={Grauert, Hans},
   title={\"Uber Modifikationen und exzeptionelle analytische Mengen},
   language={German},
   journal={Math. Ann.},
   volume={146},
   date={1962},
   pages={331--368},
   issn={0025-5831},
}
   \bib{GU}{book}{
   author={Gunning, R. C.},
   title={Lectures on Riemann surfaces},
   series={Princeton Mathematical Notes},
   publisher={Princeton University Press, Princeton, N.J.},
   date={1966},
   pages={iv+254},
}

\bib{Lo}{article}{
   author={Lojasiewicz, S.},
   title={Triangulation of semi-analytic sets},
   journal={Ann. Scuola Norm. Sup. Pisa (3)},
   volume={18},
   date={1964},
   pages={449--474}
}

%
%

 \bib{crass}{article}{
   author={Mongodi, Samuele},
   author={Slodkowski, Zbigniew},
   author={Tomassini, Giuseppe}
   title={On weakly complete surfaces},
   journal={C. R. Math. Acad. Sci. Paris},
   date={2015},
   doi={10.1016/j.crma.2015.08.009},
   }


   \bib{mst}{article}{
   author={Mongodi, S.},
   author={Slodkowski, Z.},
   author={Tomassini, G.},
   title={Weakly complete surfaces},
   journal={Indiana U. Math. J.},
   date={2016},
   note={to appear}
   }
   
%
%
%
%
%
 \bib{Sl2}{article}{
 author={Slodkowski, Zbigniew},
   title={Local maximum property and $q$-plurisubharmonic functions in
   uniform algebras},
   journal={J. Math. Anal. Appl.},
   volume={115},
   date={1986},
   number={1},
   pages={105--130},
   issn={0022-247X},
   doi={10.1016/0022-247X(86)90027-2},
}
%
%


\bib{two}{article}{
   author={Tworzewski, P.},
   author={Winiarski, T.},
   title={Continuity of intersection of analytic sets},
   journal={Ann. Polon. Math.},
   volume={42},
   date={1983},
   pages={387--393},
   issn={0066-2216},
}	

%
\bib{voi}{book}{
   author={Voisin, Claire},
   title={Hodge theory and complex algebraic geometry. I},
   series={Cambridge Studies in Advanced Mathematics},
   volume={76},
   edition={Reprint of the 2002 English edition},
   note={Translated from the French by Leila Schneps},
   publisher={Cambridge University Press, Cambridge},
   date={2007},
   pages={x+322},
   isbn={978-0-521-71801-1},
}
  \end{biblist}
\end{bibdiv}
\end{document}